\documentclass[12pt, reqno]{amsart}

\usepackage{amssymb}
\usepackage{mathtools}
\usepackage[pagebackref=true]{hyperref}
\input xy
\xyoption{all}

\makeatletter
\@namedef{subjclassname@2020}{\textup{2020} Mathematics Subject Classification}
\makeatother
\usepackage{mathrsfs}

\hypersetup{colorlinks,citecolor=blue,linktocpage}
\usepackage{amsmath, amscd,amssymb,euscript}
\usepackage{mathtools}
\usepackage{latexsym}
\usepackage{graphicx}
\usepackage{tikz}
\usetikzlibrary{backgrounds}
\usetikzlibrary{decorations.fractals}
\usetikzlibrary{calc,intersections,through,backgrounds}
\usetikzlibrary{arrows.meta}
\tikzset{In/.tip = {Hooks[right]}}
\tikzset{Onto/.tip = {To[sep] To}}
\tikzset{Eq/.style = {-, double equal sign distance}}
\tikzset{Iso/.style = {edge node = {node[above] {$\sim$}}, inner sep = 0}}
\usetikzlibrary{positioning}
\usepackage{mathrsfs}
\usepackage{epstopdf}
\usepackage{rotating}
\usepackage{lscape}

\usepackage{color}
\definecolor{alizarin}{rgb}{0.82, 0.1, 0.26}
\definecolor{darkred}{rgb}{0.55, 0.0, 0.0}
\definecolor{brightmaroon}{rgb}{0.76, 0.13, 0.28}

\newcommand{\bC}{{\mathbb C}}

\newcommand{\bP}{{\mathbb P}}
\newcommand{\bQ}{{\mathbb Q}}
\newcommand{\bR}{{\mathbb R}}

\newcommand{\bV}{{\mathbb V}}
\newcommand{\bZ}{{\mathbb Z}}

\newcommand{\cF}{{\mathcal F}}

\newcommand{\cK}{{\mathcal K}}
\newcommand{\cL}{{\mathcal L}}

\newcommand{\cO}{{\mathcal O}}
\newcommand{\cP}{{\mathcal P}}

\newcommand{\cV}{{\mathcal V}}
\newcommand{\cW}{{\mathcal W}}
\newcommand{\cX}{{\mathcal X}}

\newcommand{\onabla}{\overline{\nabla}}

\newcommand{\tp}{\widetilde{p}}

\newcommand{\ra}{\rightarrow}
\newcommand{\lra}{\longrightarrow}

\newtheorem{proposition}{Proposition}[section]
\newtheorem{lemma}[proposition]{Lemma}

\newtheorem{theorem}[proposition]{Theorem}
\newtheorem{theoremi}{Theorem}
\newtheorem{conjecture}[theoremi]{Conjecture}
\newtheorem{question}[theoremi]{Question}
\newtheorem{definition}[proposition]{Definition}
\newtheorem{corollary}[proposition]{Corollary}

\theoremstyle{remark}
\newtheorem{remark}[proposition]{Remark}

\theoremstyle{definition}

\numberwithin{equation}{section}

\newcommand{\Hom}{\operatorname{Hom}}

\newcommand{\Ker}{\operatorname{Ker}}

\newcommand{\Pic}{\operatorname{Pic}}

\newcommand{\sB}{\mathcal{B}}
\newcommand{\sC}{\mathcal{C}}

\newcommand{\sF}{\mathcal{F}}

\newcommand{\sH}{\mathcal{H}}

\newcommand{\sK}{\mathcal{K}}
\newcommand{\sL}{\mathcal{L}}

\newcommand{\sP}{\mathcal{P}}

\newcommand\sW{{\mathcal W}}
\newcommand\sX{{\mathcal X}}

\newcommand{\bbC}{\mathbb{C}}

\newcommand{\bbP}{\mathbb{P}}
\newcommand{\bbQ}{\mathbb{Q}}
\newcommand{\bbR}{\mathbb{R}}

\newcommand{\bbV}{\mathbb{V}}
\newcommand{\bbZ}{\mathbb{Z}}

\newcommand{\scrC}{\mathscr{C}}

\usepackage[capitalize]{cleveref}
\makeatletter
  \def\th@plain{
  \thm@headfont{\bfseries} 
  \thm@notefont{\itshape} 
  \itshape
}
\makeatother

\DeclareMathOperator{\Def}{Def}

\DeclareMathOperator{\reg}{reg}                  
                  
\DeclareMathOperator{\rank}{rank}

\DeclareMathOperator{\nef}{nef}
\DeclareMathOperator{\Mon}{Mon}

\usepackage{tikz-cd}
\usepackage[a4paper, left=2.9cm,right=2.9cm, bottom=3cm, top = 3cm]{geometry}
\usepackage[normalem]{ulem}

\begin{document}
\title[Some Density Results for HyperK\"ahler Manifolds]{Some Density Results for HyperK\"ahler Manifolds}

\author{Yajnaseni Dutta}
\address{Mathematisch Instituut, Universiteit Leiden, Einsteinweg 55, 2333 CC Leiden}
\email{y.dutta@math.leidenuniv.nl}
\author{Elham Izadi}

\address{Department of Mathematics, University of California San Diego, 9500 Gilman Drive \# 0112, La Jolla, CA
92093-0112, USA}

\email{eizadi@math.ucsd.edu}

\author{Ljudmila Kamenova}

\address{Department of Mathematics, Room 3-115, 
Stony Brook University, Stony Brook, NY 11794-3651, USA}

\email{kamenova@math.stonybrook.edu}

\author{Lisa Marquand}

\address{Courant Institute of Mathematical Sciences, New York University, 251 Mercer St, NY 10012, USA}

\email{lisa.marquand@nyu.edu}

\newcommand{\yd}[1]{{\color{magenta} \tt $\spadesuit$ Yagna: [#1]}}
\newcommand{\lisa}[1]{{\color{purple} \sf $\clubsuit$ Lisa: [#1]}}
\newcommand\lmedit[1]{{\color{purple} \sf #1}}
\newcommand{\lj}[1]{{\color{blue} \sf $\diamondsuit$ Ljudmila: [#1]}}

\dedicatory{}

\thanks{}

\begin{abstract}

Lagrangian fibrations of hyperk\"ahler manifolds are induced by semi-ample line bundles which are isotropic with respect to the Beauville--Bogomolov--Fujiki form. For a non-isotrivial family of hyperk\"ahler manifolds over a complex manifold $S$ of positive dimension, we prove that the set of points in $S$, for which there is an isotropic class in the Picard lattice of the corresponding hyperk\"ahler manifold represented as a fiber over that point, is analytically dense in $S$. 
We also prove the expected openness and density of the locus of polarised hyperk\"ahler manifolds that admit a nef algebraic isotropic line bundle. 
\end{abstract}

\maketitle

\section*{Introduction}
A hyperk\"ahler manifold is a simply connected compact K\"ahler manifold $X$ of dimension $2n$
admitting a unique (up to scaling) global homolomorphic symplectic 2-form $\sigma\in H^0(X,\Omega_X^2)$. The existence of such a 2-form makes $X$ Calabi--Yau in the sense that its canonical bundle is trivial. In dimension 2, these are K3 surfaces. In higher dimensions, two infinite series and two sporadic examples are known \cite{Bea83,OG99, OG03}, up to deformation.

Since any hyperk\"ahler manifold $X$ comes equipped with a holomorphic symplectic form $\sigma$, Lagrangian subvarieties of $X$, i.e.,\ subvarieties $Z\subset X$ of dimension $n=\frac{1}{2}\dim X$ such that $\sigma|_{Z^{\rm reg}} = 0$, are interesting to study. 
It is known \cite{Mat} that any morphism $f\colon X\to B$ where $B$ is a normal complex variety with $0<\dim B<2n$, is in fact a Lagrangian fibration, meaning a surjective morphism with connected fibers where the general fiber is a Lagrangian subvariety. The general fibers are in fact abelian varieties, as proved by Voisin (see \cite[Prop. 2.1]{Campana}). 
If follows from \cite[Remark 1.4]{matdeform} that the base $B$ is projective. Moreover, whenever $B$ is smooth, it is in fact isomorphic to $\bbP^n$ \cite{Hwa08, GrebLehn}. A Lagrangian fibration with base $B\cong\bbP^n$ is induced by a line bundle $\sL$ on $X$ with $c_1(\sL)^{n+1} = 0$. 

The second cohomology $H^2(X,\bZ)$ comes equipped with a unique non-degenerate quadratic form $q_X$, the Beaville--Bogomolov--Fujiki form, which controls a great deal of geometric information about the hyperk\"ahler manifold $X$. It has signature $(3, b_2 - 3),$ and there exists a non-zero constant $c_X \in \bbQ$ satisfying $\alpha^{2n} = c_X\cdot q_X(\alpha)^n$. In particular, admitting a line bundle $\sL$ with $c_1(\sL)^{n+1}=0$ is equivalent to $q_X(\sL)=0.$ Therefore, the first step in studying Lagrangian fibrations is the study of isotropic classes. 

We prove the following:

\begin{theoremi}\label{thm:densityfamily}
Let $\cX\rightarrow S$ be a non-isotrivial family of hyperk\"ahler manifolds with $b_2(X_t)\geq 5$ over a complex manifold $S$ with $\dim S>0.$ Then the set
\[\{t\in S\mid \exists \alpha\in H^{1,1}(\cX_t,\bbZ), \alpha\neq 0,  q_X(\alpha)=0\}\]
is either empty or dense in $S$ in the analytic topology. 
\end{theoremi}


%

The case when $S=\Def(X)$, the whole deformation space of $X$, is already known. 
For  hyperk\"ahler manifolds, there is a global Torelli theorem which allows one to prove such density statements directly on the period domain (i.e.,\ the type IV symmetric domain associated to $(H^2(X,\bbZ), q_X)$; see \S \ref{sec: dense isotrop} for details). On the period domain, these types of density results follow from ergodic properties \cite[Prop.\ 7.1.3]{HuyK3}. For certain moduli spaces of lattice polarised hyperk\"ahler manifolds, Mongardi and Pacienza establish similar density results in \cite[Theorem 3.15]{MongardiPacienza}. 
However, the question of density in an arbitrary family demands different methods. Our argument is Hodge theoretic and hence applies more generally to any non-isotrivial family of K3-type Hodge structures (see Theorem \ref{thm: isot dense} for a precise statement). 

One should compare Theorem \ref{thm: isot dense} to the classical density result for Noether--Lefschetz loci, which, in the set-up of the theorem, states that the locus of hyperk\"ahler manifolds where the Picard rank jumps is analytically dense in the base. This result is  attributed to Oguiso \cite[Theorem 1.1]{Oguiso} and the key ingredient (also a key ingredient in our proof) is the criterion for the density of the Hodge classes, due to Green \cite[Proposition 5.20]{VoisinHTtwo}. 
For certain families of polarized hyperk\"ahler manifolds, the density of the loci of hyperk\"ahlers with algebraic isotropic classes follows from  equidistribution results of Tayou \cite{Tayou}, and Tayou and Tholozan \cite{TT}. 


Nefness is a necessary condition for an algebraic isotropic class $\alpha$ to induce a Lagrangian fibration. The hyperk\"ahler SYZ conjecture postulates that it is in fact sufficient. The conjecture was first formulated by Strominger, Yau and Zaslow in \cite{syz} in the context of mirror symmetry of Calabi--Yau pairs. It has been reformulated for hyperk\"ahler manifolds by Huybrechts, Hasset--Tschinkel, Tyurin, Bogomolov, and Sawon.

\begin{conjecture}{\emph{(}e.g., \cite[Conj.\ 4.1, 4.2]{SawSYZ}\emph{)}}\label{conj: SYZ1}\label{conj: SYZ2}
\begin{enumerate}
    \item A hyperk\"ahler manifold $X$ has a Lagrangian fibration if and only if it admits a non-trivial isotropic nef line bundle $L$.
    \item A hyperk\"ahler manifold $X$ admits a rational Lagrangian fibration (i.e.,\ it is birational to a hyperk\"ahler manifold which admits a Lagrangian fibration) if and only if it admits a non-trivial isotropic line bundle $L$.
\end{enumerate}
\end{conjecture}
Note that in (2) the isotropic line bundle $L$ is not required to be nef. An explicit example of a hyperk\"ahler manifold with a strictly rational Lagrangian fibration satisfying (2) was  found in \cite[Thm. 4.1]{Sacbirational}, by performing a Mukai flop along the image of a section of a regular Lagrangian fibration.
Further, according to \cite[Lemma 2.18]{kam-lehn}, for a fixed hyperk\"ahler manifold $X$, then (1) is satisfied for all hyperk\"ahlers deformation equivalent to $X$ if and only if (2) is satisfied for all hyperk\"ahlers deformation to $X$.
Therefore, one may focus on isotropic nef classes. To this end, Matsushita and Kamenova--Verbitsky studied the locus $\Def(X,L)_{\reg}$ of hyperk\"ahler manifolds which admit a regular Lagrangian fibration in the deformation space $\Def(X,L),$ where $L$ is an algebraic isotropic line bundle. They showed:

\begin{theoremi}\cite[Theorem 3.4]{kam-ver},\cite[Theorem 1.2]{mat3}\label{thm:mat}
    Let $X$ be a hyperk\"ahler manifold and $L$ an isotropic line bundle on $X$. If $L$ induces a rational Lagrangian fibration, then $\Def(X,L)_{\reg}$ is a dense open subset of $\Def(X,L)$ and the set $\Def(X,L)\setminus \Def(X,L)_{\reg}$ is contained in a countable union of hypersurfaces in $\Def(X,L)$.
\end{theoremi}

In light of the SYZ conjecture, it is expected that the locus $\Def(X,L)_{\nef}$ of hyperk\"ahler manifolds which admit an algebraic isotropic nef class is also open and dense in $\Def(X,L)$. The hyperk\"ahler SYZ conjecture is now established for all known deformation classes of hyperk\"ahler manifolds. This is due to \cite[Theorem 1.5]{BM14} and \cite[Theorem 1.3, 6.3]{Mar14} for deformations of Hilbert schemes of points on a $K3$ surface, to \cite[Prop. 3.38]{Yos16} for deformations of generalized Kummer manifolds, \cite[Cor. 1.3, 7.3]{MR21} for deformations of $OG6$, and \cite[Thm 2.2]{MO22} for deformations of $OG10$. In the case of general polarised hyperk\"ahler manifolds, we show this expected openness:

\begin{theoremi}\label{thm: openness of nef}
Let $X$ be a hyperk\"ahler manifold and $L$ an isotropic nef line bundle on $X$.
    The locus $\Def(X,L)\setminus \Def(X,L)_{\rm nef}$ is a countable union of hypersurfaces in $\Def(X,L).$ Moreover, if, in addition, $X$ is a polarised hyperk\"ahler manifold with ample line bundle $H$, then $\Def(X, L,H)_{\rm nef}$ is open and dense in $\Def(X, L, H)$.
\end{theoremi}

 The countable union of hypersurfaces in the complement of the locus $\Def(X,L)_{\rm nef}$, as mentioned in Theorem \ref{thm: openness of nef}, is in fact induced by a countable union of the so-called \textit{wall divisors} (see Definition \ref{def:walldivisor}). We use the fact that on the polarised period domain, the union of \textit{all} hypersurfaces induced by these wall divisors is closed. This was shown by Amerik and Verbitsky. Roughly speaking, the closedness follows from the properly discontinous action of the group of isometries of the lattice on the period domain of polarised marked hyperk\"ahler manifolds. Our proof critically uses this - see \S \ref{sec:opennessnef}.

 Let us finally note that questions about openness of the nef locus have been addressed in the literature  beyond the hyperk\"ahler setting \cite[Theorem 1.1]{Moriwaki}. As a consequence of Theorem \ref{thm: openness of nef} above, given a family of suitably polarised hyperk\"ahler manifolds with an isotropic line bundle $L$ on it, Moriwaki's result \cite{Moriwaki} gives a bound on the minimum degree of the wall divisor that makes $L$ non-nef (see Corollary \ref{cor:moriwaki} for a precise statement).
 
\section{Density of isotropic classes}\label{sec: dense isotrop}

Throughout, we denote by $(S, \bV, \cF^\bullet, \nabla, q)$ a variation of Hodge structures (VHS), with a bilinear form $q$, over a connected complex manifold $S$. Here $\bV$ is a local system of finite dimensional $\bQ$-vector spaces with a Hodge filtration $\cF^\bullet$ of the holomorphic bundle $\cV:=\bV\otimes \cO_S$, and $\nabla\colon\cV\rightarrow \cV\otimes \Omega_S^1$ is the Gauss--Manin connection. The form $q$ is a flat bilinear form  $\bV \otimes \bV \ra \bQ_S$ whose restriction to each fiber $V_s$ of $\bV$ at a point $s\in S$ is non-degenerate on $V_s$.

\begin{definition}
    A finite dimensional vector space $V_{\bbQ}$ with a quadratic form $q$ is of K3-type if the signature of $q$ is $(3, \dim V_{\bbR} - 3)$ with $\dim V_{\bbQ}\geq 4$.
    We say that a VHS $(S, \bV, \cF^\bullet, \nabla, q)$ is of \textsl{$K3$-type} if it is a variation of weight 2 Hodge structures with $\rank \sF^2=1$ and each fiber is a K3-type vector space. 
\end{definition}

Given a K3-type vector space $(V_{\bbQ}, q)$, the period domain of K3-type Hodge structures on $(V_{\bbQ},q)$ is defined as
\[\Omega\coloneqq \{\sigma\in\bbP(V_\bbC)\mid q(\sigma, \overline{\sigma}) = 0, q(\sigma)>0\}.\]

Let $(S, \bV, \cF^\bullet, \nabla, q)$ be a VHS of K3-type  with a trivialisation $\varphi\colon \bbV\overset{\sim}{\to} V_{\bbQ}\otimes \bbQ_S$.  
We have the holomorphic period map:\[\sP\colon S\to \Omega\]
defined by $t\mapsto \varphi_t(\sF^2_t)$. We say that a VHS of K3-type is \textit{\textbf{non-isotrivial}} if the holomorphic period map $\sP$ is non-constant.

 For a class $\alpha\in V_{\bbQ}$, we denote by $\alpha^{\perp}(S)\subset S$ the locus:
 \[\alpha^\perp(S):=\{t\in S \mid \alpha\in \sF^1_t\}.\] Note that  $\alpha^\perp(S)=\{t\in S\mid \alpha \in V^{1,1}_t\}$, since $\alpha$ is equal to its complex conjugate.

 We let $Q:=\{x\mid q(x)=0\}\subset V_\bC$ be the quadric cone defined by the form $q$; we will assume that $Q$ has a nonzero rational point. Note that, if $\rank (\bV) \geq 5$, this is automatically satisfied by Meyer's theorem.

\begin{theorem}\label{thm: isot dense}
Let $(S, \bV, \cF^\bullet, \nabla, q)$ be a non-isotrivial VHS of $K3$ type with trivialisation $\varphi\colon \bV\overset{\sim}{\to}V_\bQ\otimes \bQ_S$ as above. Assume that the quadric cone $Q$ has a nonzero rational point.  
Then the union
\[
S_{\rm isot}:=\bigcup_{\substack{\alpha\in V_{\bbQ}\\ \alpha\neq 0 \\q(\alpha) =0}} \alpha^{\perp}(S)
\]
is dense in $S$ in the analytic topology.
\end{theorem}

\begin{proof}
Consider the subset
\[
Q^1:= \{ (\lambda, s) \mid \lambda\in \sF^1_s, \, q(\lambda) =0\} \subset Q \times S
\]
and denote $p \colon Q^1 \ra Q$ the restriction of the projection to the first factor. By Lemma \ref{lem: submersion} below, for any $s\in S$, the map $p$ is open around some point $(\lambda, s)\in Q^1$.  Let $\pi\colon Q^1\to S$ be the restriction of the second projection. For any $s\in S$, there is an open neighbourhood $U\ni s$ such that $\pi^{-1}(U)\subset Q^1$ contains an open neighbourhood $W$ of $(\lambda, s)$ surjecting onto $U$ where $p|_W\colon W\to Q$ is open.

Note that the induced map on the real points $p_{\bbR}\colon W(\bbR)\to Q(\bbR)$ is also open. 
The sets $Q(\bQ)$ and $Q(\bR)$ are non-empty by assumption. Therefore, since $Q(\bQ)\subset Q(\bR)$ is dense, it follows that the subset $\{ (\alpha, s) \mid \alpha\in \sF^1_s, \alpha\neq 0, q(\alpha)=0 \} \cap W \subset W$ is dense. Therefore its image $\displaystyle\bigcup_{\substack{\alpha\neq 0,\\q(\alpha) =0}} \alpha^{\perp}(U)$ in $U$ is also dense.
\end{proof}
It remains to show that the map $p:Q^1\rightarrow Q$ is a submersion. We do so below: we adapt the proof of \cite[Lemma 5.21]{VoisinHTtwo} to our specific situation.

\begin{lemma}\label{lem: submersion}
For any $s\in S$, there exists $\lambda\in Q^1_s$ such that the map $p \colon Q^1 \ra Q$ is a submersion at $(\lambda, s)$.
\end{lemma}

\begin{proof}

Denote by $\tp \colon \sF^1 \lra V_\bC$
the projection to the first factor.
By \cite[Lemma 5.21]{VoisinHTtwo}, 
the map $\tp$ is a submersion at $(\lambda, s)\in \sF^1$ if the map
\[
\onabla(\lambda) \colon T_{S,s} \ra V_s^{0,2}
\]
induced by the Gauss--Manin connection is surjective. 
The map $\overline{\nabla}$ is also the differential $d\sP$ of the period map $\cP \colon S \ra \bP(V_\bC)$. This is because
\[
\onabla\colon T_{S,s} \lra \Hom (V_s^{1,1}, V_s^{0,2})\] and we have isomorphisms $$\Hom (V_s^{1,1}, V_s^{0,2}) \simeq  \Hom ((V_s^{0,2})^*, (V_s^{1,1})^*) \simeq \Hom (V_s^{2,0}, V_s^{1,1}) \simeq T_{[V_s^{2,0}]}\Omega.$$
Since the VHS $(S,\bV, \sF^{\bullet},\nabla, q)$ is not isotrivial, this differential is not zero and hence there exists $v\in T_{S,s}$ such that $d\sP(v)\colon V_s^{1,1}\to V^{0,2}_s$ is non-trivial, hence surjective. Therefore there exists $\lambda\in V_s^{1,1}$ such that $\onabla(\lambda)$ is surjective.

For a fixed $s$, the quadric $Q^1_s\cap V^{1,1}_s\subset V^{1,1}_s$ is not contained in a hyperplane by the non-degeneracy of the quadratic form. Since the dimension of $V_s^{1,1}$ is at least 2, $Q^1_s\cap V^{1,1}_s$ cannot be contained in the linear subspace $\Ker(d\sP(v))\subset V_s^{1,1}$. Therefore, the map $\tp$ is a submersion at $\lambda\in Q^1_s$. 

Since ${Q^1} = \tp^{-1}(Q)$, we obtain that for any $s\in S$, there exists $\lambda$ such that $p$ is a submersion at $(\lambda, s)$.
\end{proof}
    We can apply Theorem \ref{thm: isot dense} to non-isotrivial families of hyperk\"ahler manifolds to recover Theorem \ref{thm:densityfamily}:
\begin{corollary}\emph{(= Theorem \ref{thm:densityfamily})}
Let $\cX\rightarrow S$ be a non-isotrivial family of hyperk\"ahler manifolds over a complex manifold $S$ with $\dim S>0$ and $b_2(X_t)\geq 5$. Then the union
\[
S_{\rm isot}:=\bigcup_{\substack{\alpha\neq 0,\\q(\alpha) =0}} \alpha^{\perp}(S)
\]
is either empty or dense in $S$ in the analytic topology.
\end{corollary}

\begin{remark}
   One could also apply Theorem \ref{thm: isot dense} to the Fano variety $F(X)$ of lines associated to a cubic 4-fold $X$. The Fano variety $F(X)$ is a projective hyperk\"ahler manifold of K3$^{[2]}$-type. By \cite[Corollary 5.24]{Huycubic}, if $F(X)$ admits a rational Lagrangian fibration, then the cubic fourfold $X$ lies on a special divisor $\scrC_d$ of the moduli of cubic fourfolds, indexed by an integer $d$ where $d/2$ is a perfect square. Said differently, Theorem \ref{thm: isot dense} implies that any \textit{non-isotrivial }family of cubic fourfolds contains a dense subset of points from $\scrC_d$ where $d/2$ is a perfect square.
\end{remark}

\begin{question}\label{rmk:failure}
    In the light of the SYZ conjecture (see Conjecture \ref{conj: SYZ1}), one may wonder whether, given a non-isotrivial family of hyperk\"ahler manifolds $\cX\to S$, the locus of algebraic isotropic nef classes is also dense. 
\end{question}

\begin{remark}
    Notice that the period domain $\Omega$ contains smooth compact Riemann surfaces of any genus, however it does not contain compact complex subspaces of dimension greater than one (see the discussion in \cite[Sections 5.6 and 5.7]{GrebSchwald}). Therefore,    
    the image $\sP(S)$ of the complex manifold $S$ inside the period domain $\Omega$ is either non-compact of any dimension, or compact of dimension one. 
\end{remark}

\section{Density of nef classes and Lagrangian fibrations}\label{sec:opennessnef}

Throughout this section, we consider a marked hyperk\"ahler manifold $X$ of a fixed deformation type.  We denote by $\varphi:(H^2(X,\bZ), q_X)\overset{\sim}{\to} \Lambda$ the marking, i.e.,\ $\varphi$ is an isometry of integral lattices. Let $L\in H^{1,1}(X,\bZ)$ be an isotropic (i.e., $q_X(L)=0$) primitive line bundle  and let $\Def(X,L)$ be the hypersurface where $L$ remains algebraic\footnote{Here we really mean the locus of $t\in \Def(X)$ where the class $\varphi_t^{-1}(\varphi(L))$ remains algebraic.} in the local Kuranishi space $\Def(X)$. Let $(\cX,\cL)\rightarrow \Def(X, L)$ be the universal family, together with a line bundle $\cL$ such that $(\cX,\cL)_{0}=(X,L)$. Note that the marking $\varphi$ of $X$ naturally defines markings $\varphi_t$ of the fibers $\cX_t$, and thus we can consider the period map $$\cP\colon \Def(X)\rightarrow \Omega,$$ where $\Omega:=\{\sigma\in \bP(\Lambda\otimes \bC)\mid(\sigma\cdot\overline{\sigma})=0,\sigma^2>0\},$ as in \S \ref{sec: dense isotrop}. It follows that:
$$\Def(X,L)=\cP^{-1}(\varphi(L)^\perp),$$ with the notation of \S \ref{sec: dense isotrop}.

We will also consider the polarised case: if, in addition, $X$ admits an ample class $H$, we consider the universal family $(\cX, \cL,\sH)\rightarrow \Def(X,L,H)$, together with an ample line bundle $\sH$ such that $(\cX,\cL,\sH)_0=(X,L,H).$ Note that $$\Def(X,L,H):=\Def(X,L)\cap \Def(X,H)\subset \Def(X);$$ this is the locus $t\in \Def(X)$ where the classes $L$ and $ H$ remain algebraic.

The aim of this section is to study the locus $\Def(X,L)_{\rm nef}\subset \Def(X,L)$ of points $t\in \Def(X,L)$ where the isotrivial primitive line bundle $\cL_t$ is nef. According to the SYZ Conjecture \ref{conj: SYZ1}, this is exactly the locus where $\cL_t$ defines a Lagrangian fibration on $X_t,$ and hence is a dense open subset of $\Def(X)$ by \cite[Theorem 1.2]{mat3}. 
We will prove Theorem \ref{thm: openness of nef}:
\begin{theorem}
Let $X$ be a hyperk\"ahler manifold and $L$ an isotropic nef line bundle on $X$.
    The locus $\Def(X,L)\setminus \Def(X,L)_{\rm nef}$ is a countable union of hypersurfaces in $\Def(X,L).$ Moreover, if in addition $X$ is a polarised hyperk\"ahler manifold with ample line bundle $H$, then $\Def(X, L,H)_{\rm nef}$ is open and dense in $\Def(X, L, H)$.
\end{theorem}
 This provides more evidence for the validity of the SYZ Conjectures. Here, $$\Def(X,L,H)_{\rm nef}\coloneqq \{t\in \Def(X,L,H)\mid \cL_t \text{ is nef on } X_t\}.$$

\subsection{Wall Divisors and the K\"ahler cones}
Denote by $\sC_X\subset H^{1,1}(X)\cap H^2(X,\bR)$ the closed positive cone of $X$, i.e., the connected component of 
\[\{x\in H^{1,1}(X)\cap H^2(X,\bR) \mid q_X(x)> 0\}\] that contains a K\"ahler class. Recall that the K\"ahler cone $\sK_X\subset \sC_X$ is the open convex cone of all K\"ahler classes on $X$. The birational K\"ahler cone $\sB\sK_X\subset \sC_X$ is the union of $f^*\sK_{X'}$, where $f:X\dashrightarrow X'$ is a birational map with $X'$ a hyperk\"ahler manifold. An isotropic primitive line bundle $L\in H^{1,1}(X)$ is nef if and only if $L\in \overline{\cK_X}$, and hence we have:
   \[
   \Def(X,L)_{\rm nef}\coloneqq \{t\in \Def(X,L) \mid \cL_t\in \overline{\sK_{X_t}} \}.\]

\begin{definition}\label{def:walldivisor}
    Let $D\in \Pic(X)$ be a primitive divisor, not necessarily effective. Then $D$ is a \textbf{wall divisor} if $q_X(D)<0,$ and $\gamma(D^\perp)\cap \sB\sK_X=\varnothing$ for all $\gamma\in \Mon^2(X)_{\rm Hdg}.$
\end{definition}
Here, $\Mon^2(X)_{\rm Hdg}$ is the subgroup of monodromy operators that preserve the Hodge structure on $H^2(X,\bbC)$ (see \cite[Defn. 1.1]{Markman}), and $D^\perp=\{\alpha\in \overline{\sC_X} \mid q_X(\alpha, D)=0\}.$ Thus a line bundle $L$ is nef on $X$ provided $q_X(L,D)\geq 0$ for all wall divisors $D\in \Pic(X).$ 

 Note that by \cite[Cor 1.6 and discussion after]{AVcone}, for a fixed deformation type, the Beauville-Bogomolov square of a wall divisor $D$ is bounded below, i.e.,\ there exists $N\in \mathbb{N}$ such that $-N\leq q_Y(D)<0$. We thus define:
\[\mathcal{W}:=\{v\in \Lambda \mid -N\leq v^2 <0, v \text{ primitive }\}.\] For any $t\in \Def(X)$ and any wall divisor $D_t\in H^{1,1}(X_t),$ we have $\varphi(D_t)\in \mathcal{W}.$

The notion of wall divisors is equivalent to that of MBM classes \cite{AVrationalcurves}. Indeed, via the BBF form $q_X$, we can embed $H_2(X,\bZ)\subset H^2(X,\bQ)$ as an overlattice of $H^2(X,\bZ)\cong \Lambda$. We can then view classes of curves in $X$ as rational $(1,1)$-classes. Note that if $D\in \Pic(X)$ is a wall divisor which is positive on a K\"ahler class, and satisfies   $D^\perp\cap\overline{\sK_X}\neq\varnothing$, then $D$ has a rational multiple represented by an extremal curve class - an MBM class \cite[Prop 1.5]{MonKahCone}.

\subsection{Proof of Theorem \ref{thm: openness of nef}}
The proof follows from the two propositions below.

\begin{proposition}\label{prop: countable hypersurf}
    In the notation of Theorem \ref{thm: openness of nef}, the locus $\Def(X,L)\setminus \Def(X,L)_{\nef}$ is a countable union of hypersurfaces in $\Def(X,L).$
\end{proposition}
    \begin{proof}
    Let $t\in \Def(X,L)\setminus \Def(X,L)_{\nef}$. Since $L_t\coloneqq \sL|_{X_t}$ is not nef on $X_t$, there exists a wall divisor $D_t\in \Pic(X_t)$ such that $D_t^\perp\cap\overline{\sK_{X_t}}\neq\varnothing$, i.e.,\ the wall $D_t^\perp$ is a wall of the K\"ahler cone, and $q_{X_t}(D_t, L_t)<0$. 
    Then, by \cite[Prop 1.5]{MonKahCone}, the class $D_t\in H_2(X_t,\bZ)$ generates an extremal ray of the Mori cone (the closed convex cone generated by the classes of effective curves), and some positive rational multiple $kD_t$ is represented by an effective curve $C\subset X_t.$ 
    By the Cone Theorem, these effective curves are rational when $X_t$ is projective; in the non-projective case, it was observed in \cite{AVrationalcurves}, that these extremal curces are the so-called MBM curves and are rational.
    
     Let $v:=\varphi_t(D_t)\in \cW$. Recall, in the notation of \S \ref{sec: dense isotrop}, that
     $$v^\perp(\Def(X,L)):=\{s\in \Def(X,L) \mid \varphi_s^{-1}(v)\in H^{1,1}(X_s)\cap H^2(X,\bZ)\}.$$
     Consider $t'\in v^\perp (\Def(X,L))$. Since the rational curve $C$ deforms exactly to the locus where $D_t$ remains of type $(1,1)$, it follows that there is a curve $C_{t'}$ (possibly non integral) on $X_{t'}$ such that  $$L_{t'}\cdot C_{t'}=k q_{X}(L_t, D_t)<0$$ (see \cite[Discussion after Defn 2.10]{AVrationalcurves}). It follows that  $L_{t'}$ is not nef on $X_{t'}$, and $v^\perp(\Def(X,L))\cap \Def(L,X)_{\nef}=\varnothing.$ Thus 
     $$\Def(X,L)\setminus \Def(X,L)_{\nef}=\bigcup\limits_{\substack{(v\cdot \varphi(L))<0\\ v\in \mathcal{W}}}v^{\perp}(\Def(X,L)),$$ a union of hypersurfaces. Since $\sW$ is countable, the union is also countable.
\end{proof}

\begin{proposition}
    Let $X$ be a polarised hyperk\"ahler manifold with ample bundle $H$. Then $\Def(X,L,H)_{\nef}$ is open and dense in $\Def(X,L,H).$
\end{proposition}
\begin{proof}
   By \Cref{prop: countable hypersurf} we see that 
    \[\Def(X,L,H)\setminus \Def(X,L,H)_{\nef} = \bigcup\limits_{\substack{(v\cdot \varphi(L))<0\\ v\in \mathcal{W}}}v^{\perp}(\Def(X,L,H)).\]
    We will show that, for any $t\in \Def(X,L,H)_{\nef}$, there is an open subset $U$ containing $t$ such that $U\subset \Def(X,L,H)_{\nef}$. 
 
    Put $h:=\varphi_t(H)$, and denote by $\Omega_h\coloneqq \Omega\cap \bbP(h^{\perp}\otimes\bbC)$, the associated period domain. We denote by $\Mon(\Lambda, h)\coloneqq \{g\in \Mon(X_t)\mid g(h) = h\}$.  
    Since the lattice $h^{\perp}\subset \Lambda$ has signature $(2, *)$, by \cite[Remark 6.1.10]{HuyK3} the group $\Mon(\Lambda, h)$ acts properly discontinuously on $\Omega_h$ and hence for any $v\in h^\perp$, the union $\bigcup_{g\in \Mon(\Lambda, h)} g(v)^{\perp}$ is a proper closed subset. By \cite[Theorem 3.7]{AVrationalcurves}, the monodromy group $\Mon(\Lambda, h)$ acts on the set of wall divisors $\cW$ with finitely many orbits. Therefore the set 
    \[\bigcup\limits_{\substack{g\in \Mon(\Lambda, h)\\ v^2<0 \\ q(v, \varphi_t(L_t))<0}} g(v)^{\perp}\]
    is closed in $\Omega_h$. The complement $U$ of this set in $\Def(X,L, H)$ is an open set containing $t$. Further, $U$ is contained in $\Def(X,L,H)_{\nef}$ since, by the universality of the Kuranishi space, we have $\Def(X_t,L_t,H_t)= \Def(X,L,H)$ for $t$ sufficienly close to $0$. 
     
     The density follows immediately since the complement of $\Def(X,L,H)_{\nef}$ is a countable union of complex analytic codimension one subsets. 
\end{proof}

As a consequence of the theorem above, the equivalent criterion for such openness due to Moriwaki \cite[Theorem 1]{Moriwaki} implies the following: 
\begin{corollary}\label{cor:moriwaki}
     Let $\sX\to S$ be a non-isotrivial family  of marked hyperk\"ahler manifolds with a fixed polarization $H$ over a smooth base scheme $S$. Suppose that $\sL$ is a line bundle on $\sX$ such that under the marking we have $\varphi_s(\sL_s)=\alpha$ for a fixed  $\alpha\in \Lambda$. Define
  \[\sigma_s\coloneqq \min\{q_{X_s}(H,D)\mid  D\text{ is a wall divisor, }  D^\perp\cap \overline{\sK}_{X_s}\neq \varnothing \text{ and }q_{X_s}(D,\alpha)< 0\}.\] Then there exists an integer $N$, such that for all $s\in S$, $\sigma_s< N$.
 \end{corollary}

 \begin{proof}
    Since, on hyperk\"ahler manifolds, we can detect nefness by intersecting with wall divisors (with respect to $q_X$), we are able to mimic Moriwaki's argument to obtain the result. Indeed, let $X = \sX_0$, then by \cite[Th\'eor\`em 5, Corollaire 1]{Bea83} there is a universal family over $\Def(X,H,L)$ and there is a map $\varphi\colon S\to \Def(X,H,L)$ such that $\sX$ is the pullback of the universal family over $\Def(X,H,L)$. By Theorem \ref{thm: openness of nef} above,  there is an open subset $U\subset S$ such that for all $s\in U$, $\sL_s$ is nef on $X_s$. In particular, for all $s\in U$, $\sigma_s = 0$. Let $S'\subset S\setminus U$ be an irreducible component. We can conclude from \cite[Lemma 2]{Moriwaki} that for all $s\in S'$ we have $\sigma_s<\sigma_{\eta_{S'}}$ where $\eta_{S'}$ is the generic point of $S'$. This concludes the result.
  
 \end{proof}

\section*{Acknowledgments} 
We are grateful to BIRS in Banff, and to the organizers of the workshop ``Women in Mathematical Physics II" in August 2023, where this project started. We would like to thank Ignacio Barros, Emma Brakkee, Daniel Huybrechts, and  Claire Voisin  for interesting conversations and comments on the results. 
We thank Christian Lehn, Giovanni Mongardi, Salim Tayou, and the referee for their useful comments. The third named author is partially supported by a grant from the Simons Foundation/SFARI (522730, LK).

\bibliographystyle{halpha}
\bibliography{main}
\end{document}